\documentclass[12pt,english,a4paper]{amsart}

\usepackage[utf8]{inputenc}

\usepackage{a4}
\usepackage{amsfonts}
\usepackage{amssymb}
\usepackage{amsmath}
\usepackage{amsthm}
\usepackage{changepage}
\usepackage{nomencl}
\usepackage{geometry}
\usepackage{tikz}
\usepackage{verbatim}
\usetikzlibrary{matrix}

\usepackage{graphicx}
\usepackage{hyperref}

\begin{document}
\providecommand{\keywords}[1]{\textbf{\textit{Keywords: }} #1}

\newtheorem{satz}{Theorem}[section]
\newtheorem{prop}[satz]{Proposition}
\newtheorem{kor}[satz]{Corollary}

\theoremstyle{remark}
\newtheorem{remark}{Remark}

\newcommand{\cc}{{\mathbb{C}}}   
\newcommand{\nn}{{\mathbb{N}}}   
\newcommand{\qq}{{\mathbb{Q}}}  
\newcommand{\rr}{{\mathbb{R}}}

\title[An approach for computing families of covers]{An approach for computing families of multi-branch-point covers and applications for symplectic Galois groups}

\author{Dominik Barth}\email{dominik.barth@mathematik.uni-wuerzburg.de}
\address{Institute of Mathematics, University of W\"urzburg, Emil-Fischer-Stra{\ss}e 30, 97074 W\"urzburg, Germany}

\author{Joachim K\"onig}\email{jkoenig@kaist.ac.kr}
\address{Department of Mathematical Sciences, KAIST, 291 Daehak-ro Yuseong-gu, Daejeon 34141, South Korea}

\author{Andreas Wenz}\email{andreas.wenz@mathematik.uni-wuerzburg.de} 
\address{Institute of Mathematics, University of W\"urzburg, Emil-Fischer-Stra{\ss}e 30, 97074 W\"urzburg, Germany}

\begin{abstract}{We propose an approach for the computation of multi-parameter families of Galois extensions with prescribed ramification type. More precisely, we combine existing deformation and interpolation techniques with recently developed strong tools for the computation of $3$-point covers. To demonstrate the applicability of our method in relatively large degrees, we compute several families of polynomials with symplectic Galois groups, in particular obtaining the first totally real polynomials with Galois group $\text{PSp}_6(2)$.}\end{abstract}
\keywords{Galois theory; explicit computation; Hurwitz spaces; Belyi maps}
\maketitle

{\small
Note: Supplementary data are contained in an extra file, available at:
\begin{center}
 \url{https://arxiv.org/src/1803.08778/anc/anc.txt}
\end{center}
}

\section{Introduction and overview of results}
This article deals with the explicit computation of families of Galois extensions of $\mathbb{Q}(t)$ with prescribed Galois group and ramification type. Techniques for such computations as well as examples of interest have been exhibited in several papers, notably \cite{Malle}, \cite{Couveignes}, \cite{Hal}, \cite{Koe}, \cite{Koe2017}. Methods used include Gr\"obner basis calculations (often referred to as ``direct methods"), complex and $p$-adic deformation methods, interpolation techniques, Riemann--Roch space computations etc. Since direct methods become expensive quite quickly as the degree and number of branch points increase, an idea used extensively was the deformation of covers with small branch point number into covers with larger branch point number, ideally reducing the computation of a cover with many branch points to computation of covers with only three branch points. While this procedure has been applied successfully in several special cases, it also has obvious downsides, notably possible problems with numerical instability as well as its length in case of iterative application. Here, we present a modified approach, reducing the computation of Galois covers with group $G$ and $r$ branch points directly to the computation of $3$-branch point covers. This comes at the cost of increasing the degree of the cover; however, in light of recent rapid improvements regarding the calculation of $3$-point covers, this is a price worth paying. 

This paper is structured as follows:
Section \ref{sec:main_ideas} introduces the key idea of reducing multi-branch point covers to those with three branch points, and on how this helps to compute the entire family of polynomials with prescribed Galois group.
The required theoretical background will be established in Section \ref{RIGP}, including the theory of Hurwitz spaces, as well as computational techniques.  
 In Section \ref{sec:Psp} we demonstrate our methods by computing a two-parameter family of polynomials with symplectic Galois group $\text{PSp}_6(2)$ of degree 28 with infinitely many totally real specializations. 
Further results are collected in the final Section \ref{sec: final}, which contains a two-parameter family of polynomials of degree 27 with $\text{PSp}_4(3)$ as Galois group, a family of degree 36 polynomials with group $\text{PSp}_6(2)$, and an example of a complex approximation of a five-branch-point covering with $\text{PSp}_6(2)$ as monodromy group.

\section{Main ideas}
\label{sec:main_ideas}

\subsection{Reducing to Belyi maps}
\label{sec:belyi}

Recall that a covering $f:X \to \mathbb{P}^1_\mathbb{C}$ is called a {\it Belyi map} if it is unramified outside of $\{0,1,\infty\}$.\footnote{We also allow less than three branch points here, to simplify notation below.}
Belyi's famous theorem asserts that every compact Riemann surface which can be defined over $\overline{\mathbb{Q}}$ admits a Belyi map to $\mathbb{P}^1$. Belyi's proof uses a clever composition of covers, successively reducing the number of branch points. It was first suggested to us by Peter M\"uller to use such an idea to reduce calculation of multi-branch point covers to Belyi maps.
This can be achieved efficiently due to the following result. 
 
\begin{prop}
\label{reduction_to_3pts}
Let $r\ge 3$, and let $C=(C_1,\dots, C_r)$ be a ramification type for the finite group $G$ with non-empty Nielsen class, i.e. there are elements $c_i \in C_i$ for $i=1,\dots, r$ such that $\prod_i c_i =1$ and $G = \langle c_1, \dots, c_r \rangle$.
Then for every Belyi map $g:\mathbb{P}^1_\cc \to \mathbb{P}^1_\cc$ of degree at least $r-2$, there exists a cover $f:X\to\mathbb{P}^1_\cc$ of type $C$ such that $g\circ f:X\to \mathbb{P}^1_\cc $  is a Belyi map. Furthermore, $r-2$ is the minimal degree with this property.
\end{prop}
\begin{proof}
Assume that $g$ is as above with $\deg(g)\ge r-2$. From the Riemann--Hurwitz genus formula, it follows that the set $g^{-1}(\{0,1,\infty\})\subset\mathbb{P}^1$ is of cardinality exactly $\deg(g)+2\ge r$.
Using Riemann's existence theorem, we may pick a cover $f:X\to \mathbb{P}^1$ of type $C$ ramified only inside $g^{-1}(\{0,1,\infty\})$. 
By construction, $g\circ f$ is then unramified outside $\{0,1,\infty\}$.
Conversely, if $\deg(g)<r-2$, then the above shows that any cover of type $C$ has to be ramified at some point outside $g^{-1}(\{0,1,\infty\})$, implying that $g\circ f$ is not a Belyi map. 
\end{proof}

The point of Proposition~\ref{reduction_to_3pts} is that, assuming that we have an efficient algorithm to compute explicit equations for Belyi maps with a prescribed ramification type, we automatically get, as a component of the resulting map $g\circ f$, an explicit equation for {\textit{some}} cover in a prescribed family with more than three branch points. The technique for computation of Belyi maps which we use (see Section \ref{sec:Psp}) has been developed by the first and third author, and has previously been applied to calculate Belyi maps with interesting monodromy groups of degree up to $280$ (see \cite{Barth2016}, \cite{BW_J1}, \cite{BW_J2}, \cite{Barth2017}). Of course, any other method that allows the computation of high degree Belyi maps works as well. Alternative techniques are discussed in \cite{Sijsling2014}, \cite{Roberts2016}, \cite{Vidunas2016}, \cite{Monien2014}, \cite{Klug2014}, \cite{Monien2017}, \cite{Monien2018}.

\subsection{Outline of our algorithm}
\label{outline}
Here, we give a brief description of our algorithmic application of Proposition \ref{reduction_to_3pts}. This algorithm can be divided into the following steps:

\subsubsection*{Finding the monodromy for a suitable Belyi map}
We are given a class $r$-tuple $C=(C_1,\dots,C_r)$ of a group $G$. We choose $g:\mathbb{P}^1\to \mathbb{P}^1$ to be the cyclic cover $g:x\mapsto x^{r-2}$, ramified only at $0$ and $\infty$. We then compute a triple $(\sigma_0,\sigma_1, \sigma_\infty)$ corresponding to a Belyi map $g\circ f$ as in Proposition \ref{reduction_to_3pts}. This can be done rather easily by observing some elementary group-theoretical invariants of such a cover $g\circ f$. In particular,

\begin{itemize}
\item[i)] its Galois group embeds naturally into the imprimitive wreath product $G\wr C_{r-2} = G^{r-2}\rtimes C_{r-2}$, with $C_{r-2}$ permuting the copies of $G$ cyclically. \item[ii)] Its monodromy $(\sigma_0,\sigma_1,\sigma_\infty)$ has the following properties: $\sigma_0^{r-2}$ (resp., $\sigma_\infty^{r-2}$) is an element of $G^{r-2}$ projecting into class $C_1$ (resp., $C_r$) in every copy of $G$; and $\sigma_1$ is an element of $C_2\times ... \times C_{r-1} \subseteq G^{r-2}$.
\end{itemize}

Indeed, the embedding of the Galois group in i) is an elementary fact in Galois theory, see, e.g., Section 14.2 in \cite{Cox}; 
in the same way, ii) follows directly from our choice of branch points and inertia group generators in Proposition \ref{reduction_to_3pts}.

Permutation triples fulfilling the above conditions are then found via computer search.

\subsubsection*{A complex approximation of a single $r$-point cover, via the Belyi map}
Using the permutation triple $(\sigma_0,\sigma_1, \sigma_\infty)$ from the first step, we calculate a complex approximation of an equation $F(u,X)=0$ (with $F\in \mathbb{C}[u,X]$) for the underlying Belyi map $g\circ f$. Using the definition of $g$, we can set $u=t^{r-2}$ and factor $F$ over $\mathbb{C}[t]$, which gives us an approximate equation for the ($r$-branch point) subcover $f:X\to\mathbb{P}^1$ as described in Proposition \ref{reduction_to_3pts}. Especially for this step, we refer to Section \ref{sec:computation} for details.

\subsubsection*{Turning a single cover into a family}
Next, we input the approximation for $f$, and use Newton's method to deform the cover $f$ by moving the branch points in small steps in $\mathbb{P}^1$, thus obtaining approximate equations for covers in our family with any desired branch point locus. This has been done, e.g., in \cite{Koe}, with explicit Magma code available in Chapter 11. It is vital for numerical stability of this method that the branch points should not be too close to each other; this makes the cyclic cover $g$ above a good choice, since $f$ then has branch points $0$, $\infty$ and $(r-2)$-th roots of unity. Next, it is important to note that the coefficients of the equations thus obtained are specializations of certain algebraic functions, parameterized by an algebraic variety called the {\textit{Hurwitz space}} of our prescribed ramification type.\footnote{For this to be true, we implicitly assume the existence of a {\textit{universal family}}; e.g., the assumption $Z(G)=\{1\}$ is sufficient for this, cf.\ the following section.} This means that any sufficiently large number of these coefficients (viewed as functions) will satisfy an algebraic dependency. Finding these dependencies is possible e.g.\ via interpolation between sufficiently many different covers, and will eventually yield defining equations for the Hurwitz space itself.

\subsubsection*{Covers defined over $\mathbb{Q}$}
Finally, after such dependencies are obtained (and, ideally, improved in order to get nice equations) for all occurring coefficients we may search for rational solutions, corresponding to rational points on the Hurwitz space, i.e. (via Theorem \ref{hurwitz_points}) to covers of the prescribed ramification type defined over the rationals.

We explain the above steps in some more detail with a concrete example in Section \ref{sec:computation}. We also refer to \cite{Koe2017} for more on these and related techniques.

\subsection{Comparison with previous approaches}
There have been many previous papers on the computation of families of covers, notably by Malle (\cite{Malle}), Couveignes (\cite{Couveignes}) and the second author (\cite{Koe}, \cite{Koe2017}).
In these, either the permutation degree of the group in question is sufficiently low to allow finding an equation for at least one cover of the prescribed ramification type ``directly" (e.g., via a Gr\"obner basis approach, or via brute-force search modulo a small prime, followed by $p$-adic lifting), or the starting point of the calculation is an $(r-1)$-branch point cover, with monodromy $(\sigma_1\sigma_2, \sigma_3,...,\sigma_r)$, which is then deformed into an $r$-point cover with monodromy $(\sigma_1,...,\sigma_r)$ --- possibly iteratively, to eventually get from a $3$-point cover to an $r$-point one.
Unless the permutation degree and the number of branch points are ``small", these methods have obvious downsides. Firstly, direct methods like the Gr\"obner basis approach become very expensive as the number of variables (which is roughly the permutation degree times $r-2$, where $r$ is the number of branch points) grows.\footnote{While it is of course not possible to give a precise bound, computations with $3$-branch points are generally considered feasible for a degree into the $20$s, and correspondingly lower degree for larger branch point number.} 
Also, for $r\ge 5$, the iterated deformation process to obtain larger number of branch points is quite time-consuming.
Finally, the complex deformation techniques turn out to be numerically rather delicate in many cases. Especially where there is no {\textit{transitive}} tuple $(\sigma_1\sigma_2,\sigma_3,...,\sigma_r)$ available, experiments showed numerically unstable behaviour in many examples. The main improvement in the use of Proposition \ref{reduction_to_3pts} --- which can be considered a ``vertical" approach (going from three to many branch points by composing covers on top of each other), compared to the ``horizontal" one of deformation and moving branch  points in $\mathbb{P}^1$ --- is to circumvent the lengthy process of deformation and get directly into the prescribed family of $r$-point covers, after which the remaining calculations are rather smooth. The obvious price is that the degree of the initial Belyi map is increased by a factor of $r-2$. However, due to the far-developed methods in computation of Belyi maps, this is (often) worth the effort.

\section{Theoretical background on Hurwitz spaces and universal families}
\label{RIGP}

Formally speaking, the objects parameterized by the equations obtained in the following sections are Hurwitz spaces of straight Nielsen classes of Galois covers, as well as their associated universal families. This interpretation will not be necessary to verify the correctness of the theorems which we eventually obtain (about Galois groups of certain polynomials); it will, however, be very useful to understand the reasoning behind the individual computational steps leading to these polynomials.
We therefore recall some basic facts about Hurwitz spaces of Galois covers. For a deeper introduction, cf.\ \cite{FV}, \cite{RW} or \cite{Voe}.

\subsubsection*{Moduli spaces of Galois covers and their components}
Let $G$ be a finite group and $r\in\nn$ be such that $G$ can be generated by $r-1$ elements. Then by Riemann's existence theorem the set of all Galois covers of $\mathbb{P}^1_\cc$ with Galois group $G$ and with exactly $r$ branch points is non-empty. The set of equivalence classes of these covers is denoted by $\mathcal{H}^{in}_r(G)$ (we refer to \cite[Section 1.2]{FV} for the precise definition of the equivalence relation).
 
The set $\mathcal{H}^{in}_r(G)$ carries a natural topological structure, and also the structure of an algebraic variety. More precisely, if $\mathcal{U}_r$ denotes the space of $r$-sets in $\mathbb{P}^1$, then the {\it branch point reference map} $\Psi: \mathcal{H}^{in}(G)\to \mathcal{U}_r$ is a dominant morphism.

If $C = (C_1,...,C_r)$ is an $r$-tuple of conjugacy classes of $G$, then the set of $[f]\in \mathcal{H}^{in}_r(G)$ whose monodromy group generators 
lie in $C_1,\dots,C_r$ up to suitable permutation of the branch points is a union of connected components of $\mathcal{H}^{in}_r(G)$, and is called the (inner) {\it Hurwitz space} of $C$, denoted $\mathcal{H}^{in}(C)$.

\begin{remark}
\label{rem:reference_map}
If the class tuple $C$ contains certain conjugacy classes more than once, then the restriction of the branch point reference map $\Psi$ to the Hurwitz space $\mathcal{H}^{in}(C)$ factors through a map $\Psi': \mathcal{H}^{in}(C) \to \tilde{U}_r$, where $\tilde{U}_r$ is a space of {\it partially ordered} $r$-sets.\footnote{Just to give an illustrative example which will re-appear in the computational sections: if $C=(C_1,C_1,C_2,C_3)$ with pairwise distinct classes $C_1$, $C_2$ and $C_3$, then $\tilde{U}_r$ is the set of all $(\{\lambda_1, \lambda_2\}, \lambda_3,\lambda_4)$, where $\{\lambda_1,\lambda_2, \lambda_3,\lambda_4\}$ is a $4$-set in $\mathbb{P}^1$.} This detail is useful for computations in view of the following paragraphs. 
\end{remark}

\subsubsection*{Group-theoretical description of Hurwitz spaces via Nielsen classes}

Elements in a given fiber under the branch point reference map $\Psi$ can be described in a purely group-theoretical way, as follows.

Let $G$ be a finite group, $r\ge 2$ and $$\mathcal{E}_r(G):=\{(\sigma_1,...,\sigma_r) \in (G\setminus\{1\})^r \mid \sigma_1 \cdot ... \cdot \sigma_r = 1, \langle \sigma_1,...,\sigma_r\rangle=G\}$$
the set of all generating $r$-tuples in $G\setminus\{1\}$ with product 1. Furthermore let $\mathcal{E}_r^{in}(G)$ be the quotient of $\mathcal{E}_r(G)$ modulo conjugating the tuples simultaneously
with elements of $G$. For any $r$-tuple $C:=(C_1,...,C_r)$ of non-trivial conjugacy classes of $G$, the {\it Nielsen class} $Ni(C)$ is defined as the set of all $(\sigma_1,...,\sigma_r)\in \mathcal{E}_r(G)$ such that for
some permutation $\pi \in S_r$ it holds that $\sigma_i \in C_{\pi(i)}$ for all $i \in \{1,...,r\}$.
The {\it straight Nielsen class} $SNi(C)$ is defined as
$\{(\sigma_1,...,\sigma_r) \in \mathcal{E}_r(G) \mid \sigma_i \in C_i \text{\ for } i=1,...,r\}$.
As above, one may also define $Ni^{in}(C)$ and $SNi^{in}(C)$.

The relevance of Nielsen classes lies in the following:
\begin{prop}
\label{prop:nielsen}
\begin{itemize}
\item[a)] Elements of $\mathcal{H}^{in}_r(G)$ with a given fiber under the branch point reference map $\Psi$ (i.e., with a given branch point set) are in a natural $1$-to-$1$ correspondence with elements of $\mathcal{E}_r^{in}(G)$.
Every Nielsen class $Ni^{in}(C)$ is a union of orbits under the action of the (Hurwitz) braid group (cf.\ \cite[Chapter III.1.1 and III.1.2]{MM} for a definition of the braid group and its action), and each orbit corresponds to a connected component of the Hurwitz space $\mathcal{H}^{in}(C)$.
\item[b)] Elements of $\mathcal{H}^{in}(C)$ in a  given fiber of the map $\Psi'$ from Remark \ref{rem:reference_map} are in $1$-to-$1$ correspondence with elements of $SNi^{in}(C)$.
\end{itemize}
\end{prop}

\subsubsection*{Rational points on Hurwitz spaces and universal families}
The following classical theorem directly links the inverse Galois problem with the existence of rational points on moduli spaces (cf.\ \cite[Cor.\ 10.25]{Voe} and \cite[Th.\ 4.3]{DF2}):
\begin{satz}
\label{hurwitz_points}
Let $G$ be a finite group with $Z(G)=1$. There is a universal family of ramified coverings $\mathcal{F}:\mathcal{T}_r(G) \to \mathcal{H}_r^{in}(G)\times \mathbb{P}^1$,
such that for each $h\in \mathcal{H}_r^{in}(G)$, the fiber cover $\mathcal{F}^{-1}(h) \to \mathbb{P}^1$ is a ramified Galois cover with group $G$.
This cover is defined regularly over a field $K \subseteq \cc$ if and only if $h$ is a $K$-rational point. In particular, the group $G$ occurs regularly as a Galois group over $\mathbb{Q}$
if and only if $\mathcal{H}_r^{in}(G)$ has a rational point for some $r$.
\end{satz}

The existence of rational points as in Theorem \ref{hurwitz_points} may sometimes be deduced from the existence of rational {\it curves} on Hurwitz spaces. Particularly in the case $r=4$, restriction to a curve on $\mathcal{H}^{in}(C)$ is generally without loss. E.g., one may restrict, via $\text{PGL}_2(\cc)$-equivalence, to the set of all covers with monodromy in $C$ and ordered branch point set $(0,1,\infty,\lambda)$, for some $\lambda\in \cc\setminus\{0,1\}$. Of course, whether this $\text{PGL}_2(\cc)$-action preserves rational points on $\mathcal{H}^{in}(C)$ depends on the precise structure of $C$; therefore one may consider covers with partially symmetrized branch point sets as well --- cf.\ Chapter III.7 in \cite{MM}, and also Section \ref{sec:Psp} below.

In short, Theorem \ref{hurwitz_points} links the existence of Galois covers defined over $K$ to the existence of $K$-points on certain curves. We also refer to these as {\textit{Hurwitz curves}}.
There are well known theoretical criteria to determine the genus of these Hurwitz curves (cf.\ e.g.\ Thm. III.7.8 in \cite{MM}, as well as our application in Proposition~\ref{prop:psp62}), essentially using the monodromy of the map $\Psi'$ from Remark \ref{rem:reference_map}. In some cases, this suffices to guarantee the existence of rational points.
 
Finally, restriction of the universal family in Theorem \ref{hurwitz_points} (in the case $Z(G)=1$) yields a universal family $\mathcal{F}:\mathcal{T} \to \mathcal{H}^{in}(C) \times \mathbb{P}^1$ of covers in the Nielsen class $Ni^{in}(C)$ (cf.\ Theorem 4.3 in \cite{DF2}). Defining polynomials for this family, whose coefficients lie in the function field of the respective Hurwitz space, are usually the main object of computations; compare the following sections.

\section{Example: Totally real \texorpdfstring{$\text{PSp}_6(2)$}{PSp(6,2)}-realizations}
\label{sec:Psp}
\subsection{Theoretical results}
A classical variant of the inverse Galois problem is the question whether, for a given finite group $G$, there exists a Galois extension $F\mid\mathbb{Q}$ with $F\subset \mathbb{R}$ such that $\text{Gal}(F\mid\mathbb{Q})\cong G$.
It is known that if every finite group is a Galois group, then also every finite group is a Galois group of such a totally real extension.

Computation of ``multi-branch-point" covers is particularly important for the computation of totally real Galois extensions.
This is due to the fact that, with very few exceptions, a $\mathbb{Q}$-regular Galois extension of $\mathbb{Q}(t)$ with totally real specializations must have at least $4$ branch points, see Example 10.2 in \cite[Chapter I]{MM}.

In the following, $\text{PSp}_6(2)$ denotes the projective symplectic group on the vector space $(\mathbb{F}_2)^6$.
First, we deduce the existence of totally real $\text{PSp}_6(2)$-realizations theoretically from known criteria.
For this, we view $PSp_6(2)<S_{28}$ in its transitive action on $28$ points. We define $C_1$ to be the (unique) conjugacy class of involutions of cycle structure $(2^6.1^{16})$ in $G$, $C_2$ to be the class of involutions of 
cycle structure $(2^{12}.1^4)$ and length $3780$, and $C_3$ to be the class of elements of order $7$ (and cycle structure $(7^4)$). By $\mathcal{H}$, we denote the Hurwitz space corresponding to the class tuple $(C_1,C_2,C_2,C_3)$, and by $\mathcal{C}$ the curve on $\mathcal{H}$ corresponding to the branch point loci $(0, 1+\sqrt{\lambda}, 1-\sqrt{\lambda}, \infty)$ with $\lambda\in \mathbb{C}\setminus\{0,1\}$.\footnote{The `symmetrization" of the second and third branch point reflects the fact that $\mathcal{H}$ has a morphism to a space $\mathcal{U}_4$ of partially ordered $4$-sets, see Remark \ref{rem:reference_map}. The branch point loci corresponding to our curve $\mathcal{C}$ are $k(\lambda)$-rational points in this space.}  

\begin{prop}
\label{prop:psp62}
The Hurwitz curve $\mathcal{C}$ is a rational curve over $\mathbb{Q}$. In particular, there exist infinitely many $\qq$-regular Galois extensions $E\mid\mathbb{Q}(t)$ of ramification type $(C_1,C_2,C_2,C_3)$. Furthermore, some of these extensions possess totally real specializations, and the degree-$28$ subfield of $E\mid\mathbb{Q}(t)$ corresponding to a point stabilizer is a rational function field.
\end{prop}
\begin{proof}
Computation e.g.\ with Magma yields that $\mathcal{C}$ is of genus $0$; more precisely, Theorem 7.8a) in Chapter III of \cite{MM} yields that the degree-$70$ cover $\mathcal{C}\to \mathbb{P}^1$
induced by the branch point reference map is ramified at three points, with inertia group generators of cycle structures $(15.12^2.9.8.7^2)$, $(3^{13}.2^{14}.1^3)$ and $(2^{35})$. Since there is, e.g., a unique cycle of length 15 in the first
permutation, \cite[Proposition III.7.9]{MM} yields that $\mathcal{C}$ is a rational genus-$0$ curve over 
$\mathbb{Q}$.\footnote{Note that all conjugacy classes $C_1$, $C_2$, $C_3$ are rational conjugacy classes of $G$, which ensures that $\mathcal{C}$ is defined over $\qq$.}

Furthermore, there are tuples $(\sigma_1, \sigma_{2,1}, \sigma_{2,2}, \sigma_3)$ in the above Nielsen class fulfilling $\sigma_3 = (\sigma_3^{-1})^{\sigma_{2,2}}$ and $\sigma_1 = (\sigma_1^{-1})^{\sigma_{2,1}}$.\footnote{In fact, 
the second equality is automatic from the first due to the product-$1$ condition.}
Theorem 10.3 in Chapter I of \cite{MM} implies that there are $G$-covers $X\to \mathbb{P}^1$ of ramification type $(C_1,C_2,C_2,C_3)$, defined over $\rr$, such that all four branch points are real and the complex conjugation in the segment
between the two branch points of class $C_2$ is induced by the identity element of $G$. Since the rational points of (the rational genus-$0$ curve) $\mathcal{C}$ are dense in the set of real points, there are also infinitely many
$G$-covers with the above property which correspond to rational points on $\mathcal{C}$, i.e., which are defined over $\mathbb{Q}$. Each of those yields a $\mathbb{Q}$-regular Galois extension $E\mid\mathbb{Q}(t)$ of ramification type $(C_1,C_2,C_2,C_3)$
such that any specialization $t_0\in \mathbb{Q}$ in the segment between the two $C_2$-branch points yields a totally real extension. Of course, Hilbert's irreducibility theorem ensures that many of these specializations preserve the Galois group $G$.

The last assertion follows from the fact that the tuples in our Nielsen class are of genus 0 and that the normalizer in $G$ of a cyclic subgroup generated by an element of $C_3$ fixes one of the $7$-cycles. This last claim
implies that in a $\mathbb{Q}$-regular $G$-extension $E\mid\mathbb{Q}(t)$ as above, the degree-$28$ subfield corresponding to a point stabilizer has a place of degree $1$ extending the $C_3$-branch point. This ensures that this field is a rational genus-$0$
function field.
\end{proof}

\subsection{Steps of the computation}
\label{sec:computation}
We now turn the theoretical result of Proposition \ref{prop:psp62} into an explicit polynomial, using the approach outlined in Section \ref{sec:main_ideas}. 

\subsubsection*{The Belyi map}
As explained before, we start by computing an approximate equation for a decomposable genus zero Belyi map $X\to \mathbb{P}^1$ such that the following holds: for a degree-$2$ subcover $Y$ of $X\to \mathbb{P}^1$, the (genus zero) cover $X\to Y$ has ramification type $(C_1,C_2,C_2,C_3)$. Then $X\to \mathbb{P}^1$ is of degree $56$ with Galois group (of the Galois closure) embedding into $\text{PSp}_6(2)\wr C_2$, and
if $(x,y,(xy)^{-1})\in S_{56}^3$ is a triple describing the ramification of this Belyi map, the elements $x$, $y$, and $(xy)^{-1}$ are of cycle structure $(14^4)$, $(2^{24},1^8)$, and  $(4^6,2^{16})$, respectively.

A computer computation shows that there are 35 triples (up to simultaneous conjugation) in $\text{PSp}_6(2)\wr C_2$ satisfying property ii) from Section \ref{outline}. Among these possible triples the permutations $(x,y,(xy)^{-1})$ can be chosen arbitrarily. We pick

\begin{align*}
x\,=\;&(1, 55, 27, 36, 8, 54, 26, 32, 4, 50, 22, 30, 2, 29)\\
&(3, 34, 6, 35, 7, 56, 28, 42, 14, 52, 24, 40, 12, 31)\\
&(5, 47, 19, 45, 17, 37, 9, 43, 15, 49, 21, 44, 16, 33)\\
&(10, 51, 23, 39, 11, 46, 18, 53, 25, 48, 20, 41, 13, 38),
\end{align*}
and
\begin{align*}
 y \,=\;& (1, 20)(3, 6)(4, 21)(5, 15)(8, 25)(9, 19)(10, 18)\\
          &(11, 24)(12, 23)(13, 22)(14, 28)(16, 26)(30, 41)\\
          &(31, 52)(32, 44)(33, 49)(35, 56)(36, 48)(37, 45)\\
          &(38, 53)(39, 40)(43, 47)(46, 51)(50, 54).
\end{align*}

We now explain in greater detail on how to obtain a genus-0 Belyi map $F: \mathbb{P}^1 \to \mathbb{P}^1 $ with ramification structure $(x,y,(xy)^{-1})$ ramified over $0$, $1$ and $\infty$. Our main goal consists of finding $c_0\in \mathbb{C}$ and monic complex polynomials $p_{4}$, $q_{8}$, $q_{24}$, $r_6$, $r_{16}$ of respective degree denoted in the index such that
\[
F = \frac{c_0 \cdot  p_4^{14}}{q_8 \cdot  q_{24}^2} = 1 + \frac{(c_0 - 1) \cdot r_{16}^2 \cdot r_6^4}{ q_8 \cdot  q_{24}^2}.
\]
In particular, $c_0 \cdot p_4^{14} -  q_8 \cdot  q_{24}^2 - (c_0 - 1) \cdot r_{16}^2 \cdot r_6^4 = 0$. By comparing coefficients, we see that the coefficients of the unknown polynomials have to solve a system of polynomial equations consisting of 59 unknowns and 56 equations. Note that this difference corresponds to the fact that Belyi maps are uniquely determined up to inner M\"obius transformations. Therefore, we may assume $p_4 = X(X^2-1)(X-\ell)$ for some $\ell \in \mathbb{C}$ in order to reduce the number of unknowns by three. Then, an approximative solution for this particular system can be computed using Newton's method.\footnote{This system also admits solutions that do not correspond to $\text{PSp}_6(2)$, see for example  \cite{Sijsling2014} for more details.} A sufficiently good initial value for this approach can be obtained by constructing the dessin of $F$, i.e. the set $F^{-1}([0,1])$. This can be achieved in the following way:

Let $a:= \text{ord}(x)$, $b:= \text{ord}(y)$, $c:= \text{ord}((xy)^{-1}),$ and
$$
\Delta:=\left< \delta_a,\delta_b,\delta_c \,|\, \delta_a^a= \delta_b^b=\delta_c^c = \delta_a\delta_b\delta_c = 1 \right>.
$$
Note that $(x,y,(xy)^{-1})$ is hyperbolic since $1/a+1/b+1/c < 1$.
We now consider the embedding $\Delta \hookrightarrow \text{PSL}_2(\mathbb{R})$ described in \cite[Proposition 2.5]{Klug2014}, where $\delta_a$ (resp. $\delta_b$) is mapped to a hyperbolic rotation around $i$ (resp. $\mu i$ for some $\mu >1$) of angle $\pi/a$ (resp. $\pi/b$).
Thus $\Delta$ acts on the upper half-plane $\mathbb{H}:= \{z\in \mathbb{C}: \text{Im}(z)>0\}$ via the natural action of $\text{PSL}_2(\mathbb{R})$ on $\mathbb{H}$, that is
$$
\text{PSL}_2(\mathbb{R}) \to \text{Aut}(\mathbb{H}): \begin{pmatrix}
\alpha & \beta \\ \gamma & \delta
\end{pmatrix} \mapsto
\left( z \mapsto \frac{\alpha z + \beta }{\gamma z + \delta} \right),
$$ 
and its fundamental domain can be chosen to be the hyperbolic kite with vertices $i,h,\mu i,-\bar{h}$ for some $h \in \mathbb{H}$. Furthermore, let $\varphi$ denote the homomorphism from $\Delta$ onto $G:=\langle x,y \rangle$ such that $\delta_a \mapsto x$ and $\delta_b\mapsto y$.
With the notation $\Gamma := \varphi^{-1}(\text{Stab}_G(1))<\Delta$ we define
$$\Phi:\mathbb{H}/\Gamma \to \mathbb{H}/\Delta, \; z \text{ mod } \Gamma \mapsto z \text{ mod } \Delta. \footnote{For a topological space $X$ and a group $G$ acting on $X$ let $X/G$ denote the corresponding orbit space.}$$
This is a three-point branched cover of degree $56$  
 with monodromy group isomorphic to $G$, see for example \cite{Klug2014} for more details. Because $\mathbb{H}/\Delta$ is homeomorphic to $\mathbb{P}^1$ we may assume that the ramification locus of $\Phi$ is given by $\{0,1,\infty\}$, i.e.,$\Phi$ is a Belyi map. We now study the dessin of $\Phi$, i.e. the set $\Phi^{-1}([0,1])$ which is visualized in Figure \ref{1}. 

\begin{figure}
\begin{center}
\includegraphics[clip, trim=0.5cm 7cm 0.5cm 6cm, width=.75\textwidth]{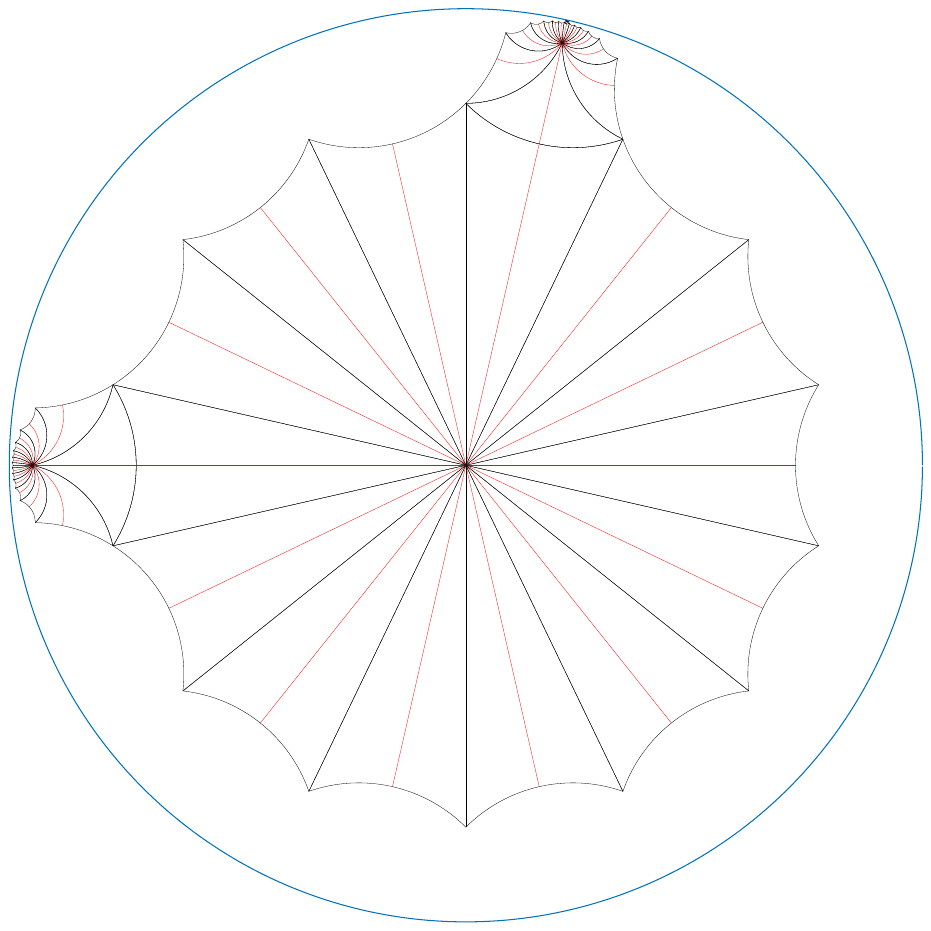}
\caption{
fundamental domain and dessin of $\Phi$ transformed to the unit disk
}
\label{1}
\end{center}
\end{figure}

Pick a connected fundamental domain $D$ of $\mathbb{H}/\Gamma$, see Figure \ref{1}. This is possible, because $G= \left< x,y \right>$ is transitive. Since $D$ is irregularly shaped, we conformally map it to $\mathbb{H}$. For example, this can be achieved by using the Schwarz--Christoffel Toolbox \cite{Driscoll1996} for MATLAB \cite{MATLAB}, see Figure \ref{2}.

\begin{figure}
\begin{center}
\includegraphics[clip, trim=0.5cm 13cm 0.5cm 11cm, width=.65\textwidth]{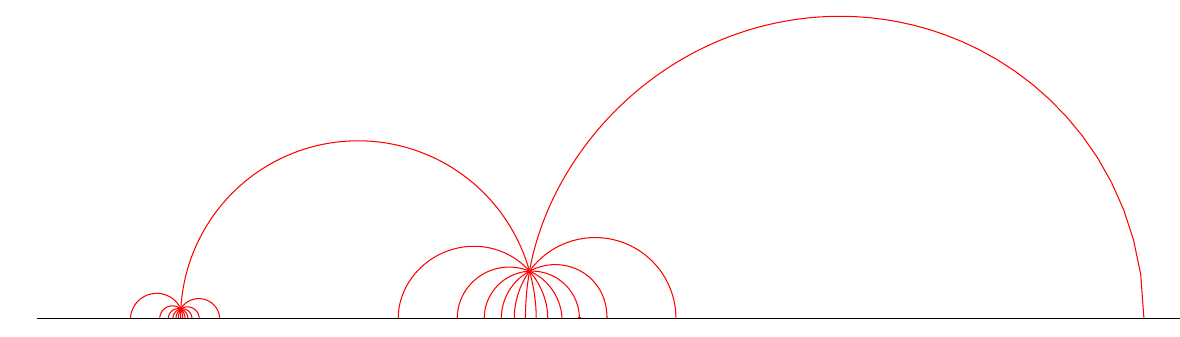}
\caption{dessin of $\Phi$ conformally mapped to $\mathbb{H}$}
\label{2}
\end{center}
\end{figure} 

Now $\partial \mathbb{H} = \mathbb{R} \cup \{\infty\}$ inherits the structure of $D$ induced by the quotient $\mathbb{H}/\Gamma$.
In order to glue corresponding (adjacent) real line segments we apply slit maps of type (Figure \ref{4})
$$
\text{slit}_{A}: \mathbb{H} \to \mathbb{H}: z \mapsto  (z-A)^A(z+1-A)^{A-1}
$$
where $0<A<1$ is a real number, see \cite{Marshall2007} and \cite{Barnes2014} for a thorough analysis. This step requires the permutation triple $(x,y,(xy)^{-1})$ to be of genus zero, otherwise corresponding line segments will not appear to be adjacent to each other at some point.

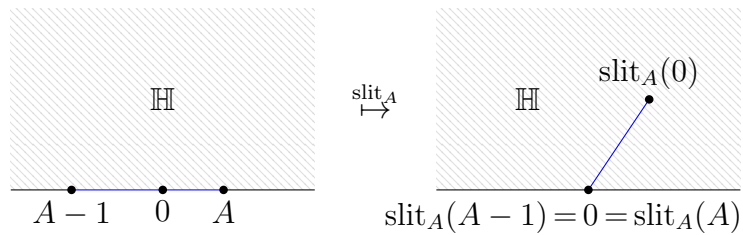
\begin{figure}
\begin{center}
\hspace{5mm}
\begin{tikzpicture}[scale = 0.8]
\usetikzlibrary{patterns}

\draw[draw = none, pattern=north west lines, pattern color=lightgray, opacity = 0.7] (-2.5,0) rectangle (2.5,3);

\coordinate[label = below:$A-1$](1-A) at (-1.5,0);
\coordinate[label = below:$0$](0) at (0,0);
\coordinate[label = below:$A$](A) at (1,0);

\draw (-2.5,0) -- (1-A);
\draw (A) -- (2.5,0);

\draw[blue] (A) -- (1-A);

\fill (A) circle (2pt);
\fill (0) circle (2pt);
\fill (1-A) circle (2pt);

\coordinate[label = center:$\mathbb{H}$](H) at (0,1.5);

\coordinate[label = center:$\stackrel{\text{slit}_A}{\mapsto}$](mapsto) at (3.5,1.5);


\draw[draw = none, pattern=north west lines, pattern color=lightgray, opacity = 0.7] (4.5,0) rectangle (9.5,3);
\coordinate[label = below:$\text{slit}_A(A-1) \,\text{=}\, 0  \,\text{=}\,\text{slit}_A(A)\phantom{--}$](E) at (7,0);

\coordinate[label = above:$\text{slit}_A(0)$](F) at (8,1.5);
\coordinate[label = center:$\mathbb{H}$](H) at (6,1.5);

\draw (4.5,0) -- (9.5,0);

\draw[blue] (7,0) -- (8,1.5);

\fill (E) circle (2pt);
\fill (F) circle (2pt);

\end{tikzpicture}
\caption{conformal map $\text{slit}_{A}$ and its behaviour on $\partial \mathbb{H}$}
\label{4}
\end{center}
\end{figure} 
 
We end up with an approximate dessin in $\mathbb{P}^{1}$ with the same ramification structure as $\Phi$, see Figure \ref{3}. With respect to the coordinates and multiplicities of the zeroes, ones and poles of the constructed approximate dessin we pick $c_0\in \mathbb{C}$ and complex polynomials $p_{4}$, $q_{8}, q_{24}$, $r_6,r_{16}$ accordingly. This gives a starting point for a successful application of Newton's method, allowing us to find an approximation of the desired Belyi map $F$ of sufficiently high precision. Note that our approach does not require to recognize the coefficients of $F$ as algebraic numbers.

\begin{figure}
\begin{center}
\includegraphics[clip, trim=0cm 7cm 0cm 7cm, width=.6\textwidth]{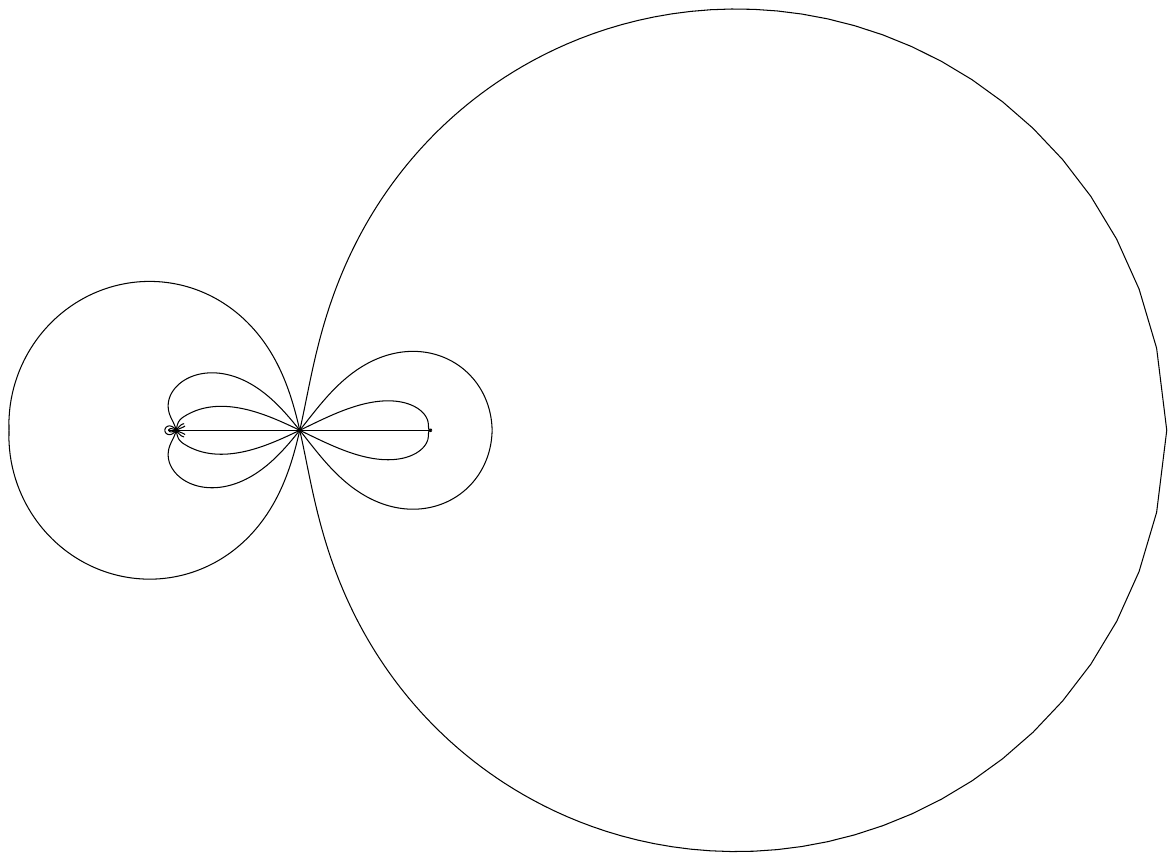}
\caption{resulting approximate dessin in $\mathbb{P}^1$}
\label{3}
\end{center}
\end{figure} 

\subsubsection*{Towards the universal family}
In the following, we describe how to obtain an equation for the entire universal family  starting from a single member. This part of the computation is relatively ``routine", see e.g.\ previous occurrences in \cite{Koe2017}.

We define $\mathcal{T}_\mathcal{C}\to \mathcal{C}\times \mathbb{P}^1$ to be the restriction of the universal family to the Hurwitz curve $\mathcal{C}$. We choose a parameter $t$ for $\mathbb{P}^1$, and a parameter $\alpha$ for $\mathcal{C}$ (which by Proposition \ref{prop:psp62} is also birational to $\mathbb{P}^1$). Recall that every $h\in \mathcal{C}$ yields a fiber cover $\mathcal{T}_\mathcal{C}(h)\to \mathbb{P}^1$ with group $G$, and due to the last assertion of Proposition \ref{prop:psp62}, this is the Galois closure of a (degree-$28$) genus-$0$ cover $\mathbb{P}^1\to \mathbb{P}^1$.

Up to M\"obius transformation, the computed Belyi map $F$ of degree $56$ is of the form $f_0(X)^2$ with a rational function $f_0$ of degree $28$, see section A in the extra file; this function has four branch points, group $\text{PSp}_6(2)$ and ramification type $(C_1,C_2,C_2,C_3)$. 
Up to rescaling, the branch points of $f_0$ are $0$, $1$, $-1$ and $\infty$.
We use this function $f_0$ as a starting point to compute the Hurwitz curve $\mathcal{C}$ of all covers with ramification type $(C_1,C_2,C_2,C_3)$ 
and branch points $0, 1+\sqrt{\lambda}$, $1-\sqrt{\lambda}$, $\infty$, with $\lambda\in \qq\setminus\{0,1\}$, by assembling equations for many fiber covers $\mathcal{T}_\mathcal{C}(h)\to \mathbb{P}^1$ in our universal family.\\
 This can be done by complex deformation techniques, slowly increasing the parameter $\lambda$ and using Newton approximation to adjust the coefficients of $f_\lambda$. The point is that this is numerically stable, since the starting point is already a cover with $4$ sufficiently separated branch points, leading to a sufficiently ``far from singular" matrix in Newton's method.\footnote{Of course, this observation should be understood as a qualitative statement, not something that can be quantified in generality.} 
The function field of the Hurwitz curve is a rational field $\mathbb{Q}(\alpha) $ by Proposition \ref{prop:psp62}, containing $\qq(\lambda)$ (where $\lambda$ is viewed as a transcendental) via the branch point reference map. Thus, the universal family $\mathcal{T}_\mathcal{C}\to \mathcal{C}\times \mathbb{P}^1$ can be parameterized by $f(t,X):=f_1(X)^2f_2(X)-tf_3(X)^7$, where $\deg(f_1)=6$, $\deg(f_2)=16$, $\deg(f_3)=3$, 
 and where each coefficient of the $f_i$ is an element of $\mathbb{Q}(\alpha)$ --- hence fulfilling some (yet unknown) algebraic dependency with $\lambda$. Analogously, the ramification indices at $1\pm \sqrt{\lambda}$ yield conditions on the factorization of $f(1\pm \sqrt{\lambda},X)$. Since a generator of the (rational!) root field of $f(t,X)$ is only unique up to $\text{PGL}_2$-action, we may apply linear transformations in $X$ and, e.g., fix the coefficient at $X^2$ in $f_3$ to be $0$ and the one at $X^5$ in $f_1$ to be $1$.\\
 In fact, two ``random" coefficients can usually be expected to generate the entire function field $\mathbb{Q}(\alpha)$.  We use this and let the coefficient $\beta$ at $X^4$ in $f_1$ converge to a rational value, using Newton approximation. 
 Then any further coefficient $\gamma$ will converge to an algebraic number of degree at most $[\qq(\beta,\gamma):\qq(\beta)]$ (which is bounded by the index of $\qq(\beta)$ in the function field of the Hurwitz curve), and given a sufficient complex precision, we can recognize this algebraic number using the LLL algorithm. 
 Doing this for several rational values of $\beta$, we obtain (by interpolation) the algebraic dependency between $\beta$ and $\gamma$ parameterizing the function field of our Hurwitz curve.

It is now easy to find the yet unknown parameter $\alpha$ of $\qq(\beta,\gamma)=\qq(\alpha)$ using a Riemann--Roch space computation.

Finally, we use Newton approximation again to let $\alpha$ converge to several rational values. Then all the coefficients of $f_1$, $f_2$, $f_3$ must also be rational values. These can easily be recognized from their complex approximation,
 and interpolating between these several values once again yields dependencies between $\alpha$ and each coefficient, i.e.\ expressions of each coefficient as a rational function in $\alpha$.

We then obtain a polynomial $f=f(\alpha,t,X)\in \qq[\alpha,t,X]$ whose Galois group over $\qq(\alpha,t)$ is isomorphic to $\text{PSp}_6(2)$. Since the coefficients of $f$ are too large to present it here,
we refer to section B in the extra file.
\begin{satz}\label{totally_real}
The polynomial $f(\alpha,t,X)=p(\alpha,X)-tq(\alpha,X)\in \mathbb{Q}(\alpha,t)[X]$, where $p$ and $q$ are given in the extra file (section B), 
has Galois group $\text{\emph{PSp}}_6(2) \leq S_{28}$ over $\qq(\alpha, t)$ and possesses infinitely many totally real specializations. The ramification with respect to $t$ is of type $(2^6.1^{16}$, $2^{12}.1^4$, $2^{12}.1^4$, $7^4)$.
\end{satz}

\begin{remark}
The Galois group of $f$ can be checked by applying the \texttt{GaloisGroup} command in Magma (after a suitable specialization and mod $p$ reduction). Although this command does not return proven results we take the output and verify its correctness rigorously in the following proof. Alternative standard techniques (such as computing the monodromy via path lifting)  only provide numerical evidence.
\end{remark}

\begin{proof}
Let $f_2,p_2,q_2\in \mathbb{Q}(t)[X]$ denote the specializations of $f,p,q$ at the place $\alpha \mapsto 2$, and $\bar{f}_2,\bar{p}_2,\bar{q}_2$ their images in $\mathbb{F}_{31}(t)[X]$ under the canonical projection.

By computing the discriminant $\Delta$ of $f$ we see that $f$ and $f_2$ have exactly four branch points with respect to $t$. Furthermore the branch cycle structure of $f$ can be derived by inspecting the inseparability behaviour of $f$ evaluated at the places $t \mapsto 0$, $t \mapsto \infty$, and $t \mapsto r_i$ for $i=1,2$ where $r_1$ and $r_2$ denote the non-zero roots of $\Delta \in \mathbb{Q}(\alpha)[t]$.

By a result of Malle (see \cite[Lemma 3.1]{Malle}), the two groups $\text{Gal}(f\mid \mathbb{Q}(\alpha,t))$ and $\text{Gal}(f_2\mid \mathbb{Q}(t))$ coincide. It remains to show that $\text{Gal}(f_2\mid \mathbb{Q}(t))$ is isomorphic to $\text{PSp}_6(2)$.

Since $\frac{1}{X-t}\cdot \bar{f}_2\left(\frac{\bar{p}_2(t)}{\bar{q}_2(t)},X\right)\in \mathbb{F}_{31}(t)[X]$ is irreducible the Galois group of $\bar{f}_2$ over $\mathbb{F}_{31}(t)$ must be 2-transitive of permutation degree 28, implying $\text{Gal}(\bar{f}_2\mid \mathbb{F}_{31}(t)) \in \{\text{PSp}_6(2),A_{28},S_{28}\}$ due to the classification of finite 2-transitive groups. Dedekind reduction yields $\text{Gal}(f_2\mid \mathbb{Q}(t)) \in \{\text{PSp}_6(2),A_{28}, S_{28}\}$.
Because both discriminants of $f_2$ and $\bar{f}_2$ are squares, $\text{Gal}(f_2\mid \mathbb{Q}(t))$ and $\text{Gal}(\bar{f}_2\mid \mathbb{F}_{31}(t))$ are not $S_{28}$. In particular, $\text{Gal}(f_2\mid \mathbb{Q}(t))$ is simple, and the corresponding function field extension must be regular, allowing us to apply a theorem of Beckmann, see \cite[Chapter I, Proposition 10.9]{MM}, to obtain $\text{Gal}(f_2\mid \mathbb{Q}(t)) \cong \text{Gal}(\bar{f_2}\mid \mathbb{F}_{31}(t))$.

Let $r(t,X)\in \mathbb{F}_{31}(t)[X]$ be the irreducible polynomial of degree 63 in the ancillary file (see section C), then $r\left( \frac{\bar{p}_2(t)}{\bar{q}_2(t)},X\right)$ becomes reducible over $\mathbb{F}_{31}(t)$.\footnote{The polynomial $r$ was obtained by using the Magma command \texttt{GaloisSubgroup} for the index 63 subgroup of $\text{PSp}(6,2)$.
} This guarantees the existence of an index $d\neq 1$ subgroup of $\text{Gal}(\bar{f_1}\mid \mathbb{F}_{31}(t))$ where $d$ is a divisor of $63$. Since $A_{28}$ does not contain such a subgroup we end up with 
$\text{Gal}(\bar{f_2}\mid \mathbb{F}_{31}(t)) = \text{PSp}_6(2)$, thus $\text{Gal}(f_2\mid \mathbb{Q}(t))  =  \text{PSp}_6(2)$.

Finally, we specialize $\alpha\mapsto 0$ (which does not decrease the number of branch points) and verify that for some specialization of $t$ in the interval $[-2.8\cdot 10^{12}, 0]$ (the left bound being approximately the only negative branch point of $f(0,t,X)$), the number of real roots of $f(0,t,X)$ is equal to $28$. The same then follows for {\textit{all}} specializations $t\mapsto t_0$ in that interval; indeed, it is elementary that the number of real preimages of $t\in \mathbb{R}$ under a rational function $r(X)\in \mathbb{R}[X]$ can only change (as a function in $t$) at a critical value of $r(X)$, i.e., at a branch point of the corresponding function field extension.
\end{proof}

 In particular, specializing  $\alpha\mapsto 0$ and applying some linear transformations to decrease the coefficients, we obtain the following:
\begin{kor}
Let {\footnotesize{$\tilde{f}(t,X):=\\
(X^6 - 33/2X^5 - 42924X^4 - 1525664X^3 + 477587712X^2 + 40478785536X +
        863547424768)^2\cdot
			\\ (X^{16} + 271X^{15} - 430719/4X^{14} - 35366300X^{13} + 3314214496X^{12} +
        1797598385556X^{11} + \\ 28249865746816X^{10} - 42517539693978944X^9 -
        3546884171151604080X^8 +\\ 388165289642365195520X^7 +
        67637298931930365811712X^6 + 1157375979002203859189760X^5 -
        370365044650038661036441600X^4 - 30197279842907494819422011392X^3 - \\
        814830488568960744917173272576X^2 + 162666689511335341711909978112X +
        \\256038325580946715804749139017728)
				-t(X^3 - 21952X - 1229312)^7\in \mathbb{Q}(t,X)$.}}\\
Then every specialization of $t$ in the interval $[-4.9\cdot 10^{13}, 0]$ which preserves the Galois group yields a totally real $\text{\emph{PSp}}_6(2)$-polynomial.
\end{kor}

\section{Further examples}
\label{sec: final}
\subsection{A two-parameter family of polynomials with Galois group \texorpdfstring{\boldmath{$\text{PSp}_4(3)$}}{PSp(4,3)}}
As a further application of our algorithm, we calculate families of polynomials $f(\alpha,t,X)$ with Galois groups $\text{PSp}_4(3)$, resp.\ $\text{PSp}_4(3).2$.
The group $G:=\text{PSp}_4(3).2$ happens to possess a {\textit{rigid}} genus-$0$ four-tuple of rational conjugacy classes. More precisely, if $G$ is viewed in its 
transitive permutation action on $27$ points, the tuple $(C_1,C_1,C_2,C_3)$ is rigid, where $C_1$ denotes the class of involutions of cycle structure 
$(2^6.1^{15})$, $C_2$ denotes the class of cycle structure $(4^6.1^3)$, and $C_3$ denotes the class of length 720 whose elements have cycle structure $(6^4.3^1)$.
The rigidity criterion (see, e.g., \cite[Theorem I.4.8]{MM}) then yields that for every $4$-set of rational points $p_1,...,p_4$, there exists a 
$\qq$-regular Galois extension of $\mathbb{Q}(t)$ with ramification type $(C_1,C_1,C_2,C_3)$ and branch points $p_1,...,p_4$.
Below, we turn this theoretical result into an explicit (and particularly nice!) parametric family of polynomials with group $\text{PSp}_4(3).2$. This family and the corresponding Belyi map of degree $54$ is given in section D in the extra file.
\begin{satz}
The polynomial $f(\alpha,t,X):=(2X^6 - 10\alpha X^4 + 10\alpha X^3 - 10\alpha^2X^2 + 2\alpha^2X + 2\alpha^3 - \alpha^2)^4(4X^3 - 4\alpha X + \alpha) - t(3X^4 - 6\alpha X^2 + 3 \alpha X - \alpha^2)^6\in\qq(\alpha,t)[X]$ 
has regular Galois group $\text{\emph{PSp}}_4(3).2\leq S_{27}$ over $\qq(\alpha,t)$, and branch cycle structure $(2^6.1^{15}$, $2^6.1^{15}$, $4^6.1^3$, $6^4.3^1)$ with respect to $t$.
Furthermore, the polynomial $f(3\alpha^2, t(s), X)\in\qq(\alpha,t)[X]$ where $t(s):=\frac{((-8\alpha/3)^3 + (-8\alpha/3)^2)(3s^2 - (\alpha + 3/8)/(\alpha - 3/8))}{3s^2+1}$ has Galois group
$\text{\emph{PSp}}_4(3)$ over $\qq(\alpha,s)$.
\end{satz}
\begin{proof}

Define $f_1,f_2\in \qq(\alpha)[X]$ such that $f=f_1-tf_2$.
The branch cycle structure can be easily computed by observing the inseparability behaviour of $f_1$, $f_2$ and of the specialized polynomials $f(\alpha,t_0,X)$, where $t_0\in \mathbb{Q}\setminus\{0\}$ is a root of the discriminant
$\Delta(\alpha,t)\in \qq(\alpha)[t]$ of $f$.
Next, computer calculation shows that $f_1(X)f_2(Y)-f_2(X)f_1(Y)$ is reducible in $\qq(\alpha)[X,Y]$, with three factors of $X$-degree $1$, $10$ and $16$.
This means that $f$ factors into degree $1$, $10$ and $16$ over $\qq(\alpha)(y)$, where $y$ is a root of $f$. In other words, the point stabilizer in ${\textrm{Gal}}(f\mid\qq(\alpha,t))$ has three orbits on the roots of sizes $1$, $10$ and $16$. This implies ${\textrm{Gal}}(f\mid\qq(\alpha,t))$ is a primitive group. Indeed, if it were imprimitive, then the stabilizer of a point would also fix a block (of size a non-trivial divisor of $27$), and hence a suitable union of orbits of this point stabilizer would have
to be of cardinality a non-trivial divisor of $27$. One now verifies that $\text{PSp}_4(3)$ and $\text{PSp}_4(3).2$ are the only primitive groups of degree $27$ with stabilizers having orbits of size $1$, $10$ and $16$. However, the inseparability behaviour
at one of the finite, non-zero branch points shows that ${\textrm{Gal}}(f\mid\qq(\alpha,t))$ contains an element of cycle structure $(2^6.1^{15})$, which $\text{PSp}_4(3)$ does not. This shows ${\textrm{Gal}}(f\mid\qq(\alpha,t))\cong \text{PSp}_4(3).2$.

Next, note that only the conjugacy class $C_1$ lies outside the index-$2$ normal subgroup $\text{PSp}_4(3)$. Furthermore, upon replacing $\alpha$ by $3\alpha^2$, 
computation of the discriminant of $f$ shows that the two branch points with inertia group generator in $C_1$ become $\mathbb{Q}(\alpha)$-rational, say 
$t\mapsto k$ and $t\mapsto \ell$ (with $k,\ell \in \qq(\alpha)$). The quadratic extension corresponding to the fixed field of $\text{PSp}_4(3)$ is thus 
ramified at exactly two points, both rational. It is therefore a rational function field, given by an equation $cY^2=(t-k)(t-\ell)$ 
(with some constant $c$), and a fractional linear transformation easily yields a parameter $s$ for such a function field, providing the equation 
$t=t(s)$ as above. 
\end{proof}

\subsection{Another two-parameter family with group \texorpdfstring{\boldmath{$\text{PSp}_6(2)$}}{PSp(6,2)} of degree 36} 
Let $(C_1,C_2,C_2,C_3)$ be the class vector of the group $\text{PSp}_6(2)$ acting 2-transitively on 36 elements where the conjugacy classes $C_1$, $C_2$ and $C_3$ are unique of type $(3^{12})$, $(1^{12}.2^{12})$, and $(1^6.2.4^7)$. In the same fashion as before one can show theoretically that $\text{PSp}_6(2)$ occurs as a Galois group over $\mathbb{Q}(\alpha,t)$ where the ramification type is given by the above class vector. Surprisingly, the explicit two-parameter family turns out have rather small coefficients:

\begin{satz} 
Let $f(\alpha,t,X) = p(\alpha,X)-tq(\alpha,X) \in \mathbb{Q}(\alpha,t)[X]$
where 
\begin{align*}
p(\alpha,X) =\; 
& \left( X^{12} + X^{11} + \left(144\alpha + \frac{1}{8} \right)X^{10} + 40\alpha X^9 + \left(-1728 \alpha^2 + \frac{21}{4}\alpha\right)X^8 \right. \\
&+ \left(-576\alpha^2 + \frac{3}{8}\alpha\right)X^7 - 84\alpha^2X^6 - 6\alpha^2X^5 + \left(144\alpha^3 - \frac{3}{64}\alpha^2\right)X^4  \\
& \left. + \;40\alpha^3X^3 + \frac{13}{4}\alpha^3X^2 + \frac{1}{8}\alpha^3X + \alpha^4 \right)^3 ,
\end{align*}
and
\begin{align*}
q(\alpha,X) =\; & \left( X^6 - 12\alpha X^4 + \frac{1}{2}\alpha ^2 \right) \cdot \left(X^3 - 24\alpha X - 2\alpha \right)^4\\
& \cdot \left(X^4 + \frac{1}{6}X^3 + \frac{1}{24}\alpha \right)^4.
\end{align*}
Then the Galois group of $f$ over $\mathbb{Q}(\alpha,t)$ is isomorphic to $\text{\emph{PSp}}_6(2)\leq S_{36}$, and the branch cycle structure of $f$ with respect to $t$ is given by $(3^{12}$, $1^{12}.2^{12}$, $1^{12}.2^{12}$, $1^6.2.4^7)$.
\end{satz}

\begin{proof}
One can mimic the proof of Theorem~\ref{totally_real} with some minor changes.
\end{proof}

The degree-$36$ polynomial appearing in the theorem as well as the corresponding degree-$72$ Belyi map are also listed in the ancillary file, see section E.

\subsection{A 5-branch point cover for \texorpdfstring{\boldmath{$\text{PSp}_6(2)$}}{PSp(6,2)}}
The advantage of our approach compared to previous ones increases as the number of branch points grows.
We give just one example of a complex approximation for a $5$-branch point cover with Galois group $\text{PSp}_6(2)$. Using the techniques of the previous sections, one could again use this to obtain an equation for a family of covers. This time, we take a genus-$0$ tuple of type $(2^{10}.1^{16}$, $2^{12}.1^{12}$, $2^{12}.1^{12}$, $2^{12}.1^{12}$, $3^{12})$ in the 2-transitive degree-$36$ permutation action of $\text{PSp}_6(2)$. Using Proposition \ref{reduction_to_3pts}, we turn this into a Belyi function of degree $108$, with imprimitive Galois group contained in $\text{PSp}_6(2)\wr C_3$, by composing with the rational function $x\mapsto x^3$, see Figure \ref{108}. The third root of this Belyi map then gives the desired 5-branch point $\text{PSp}_6(2)$-cover.  We have included it in the file, see section F.

The monodromy of the computed complex cover can be checked numerically with the path lifting algorithm in \cite[Chapter 11.1]{Koe}.

\begin{figure}
\begin{center}
\includegraphics[clip, trim=0.5cm 7cm 0.5cm 6cm, width=.7\textwidth]{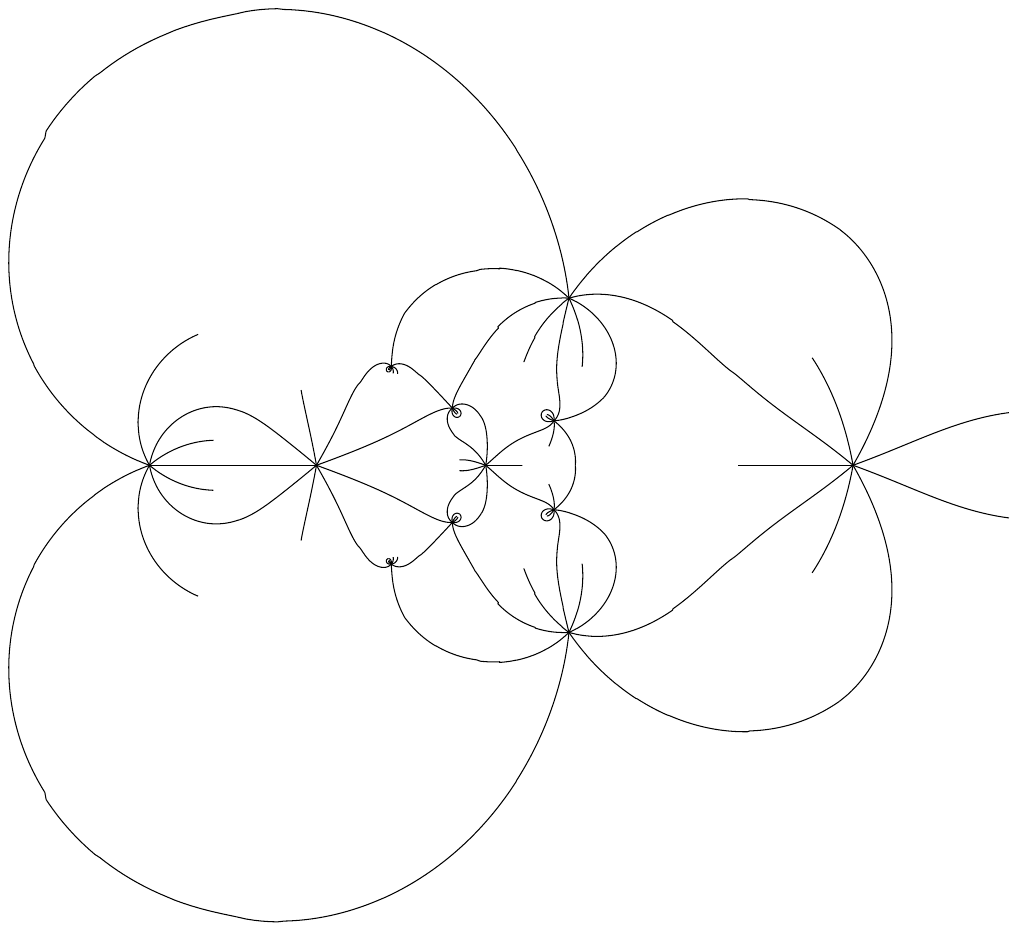}
\caption{
approximate dessin of degree 108 in $\mathbb{P}^1$
}
\label{108}
\end{center}
\end{figure}

\section*{Acknowledgements}
We are indebted to Peter M\"uller for several valuable suggestions. Thanks also to the anonymous referees for helpful feedback.

\end{document}